\documentclass[final,nomarks]{dmtcs-episciences}

\usepackage[utf8]{inputenc}
\usepackage{subfigure}

\usepackage{amsmath}

\newtheorem{theorem}{Theorem}
\newtheorem{corollary}[theorem]{Corollary}

\newtheorem{proposition}[theorem]{Proposition}

\newtheorem{remark}[theorem]{Remark}

\newcommand{\Tnw}{{T}^{\mbox{\tiny$\scriptstyle\nwarrow$}}}
\newcommand{\Tsw}{{T}^{\mbox{\tiny$\scriptstyle\swarrow$}}}
\newcommand{\Tne}{{T}^{\mbox{\tiny$\scriptstyle\nearrow$}}}
\newcommand{\Tse}{{T}^{\mbox{\tiny$\scriptstyle\searrow$}}}

\newcommand{\Pnw}{{P}^{\mbox{\tiny$\scriptstyle\nwarrow$}}}

\newcommand{\Pse}{{P}^{\mbox{\tiny$\scriptstyle\searrow$}}}
\newcommand{\tPne}{\widetilde{P}^{\mbox{\tiny$\scriptstyle\nearrow$}}}
\newcommand{\tPnw}{\widetilde{P}^{\mbox{\tiny$\scriptstyle\nwarrow$}}}
\newcommand{\tPse}{\widetilde{P}^{\mbox{\tiny$\scriptstyle\searrow$}}}

\newcommand{\Cnw}{{C}^{\mbox{\tiny$\scriptstyle\nwarrow$}}}

\newcommand{\tCse}{\widetilde{C}^{\mbox{\tiny$\scriptstyle\searrow$}}}
\newcommand{\tCsw}{\widetilde{C}^{\mbox{\tiny$\scriptstyle\swarrow$}}}
\newcommand{\cTnw}{\mathcal{T}^{\mbox{\tiny$\scriptstyle\nwarrow$}}}

\newcommand{\cTne}{\mathcal{T}^{\mbox{\tiny$\scriptstyle\nearrow$}}}
\newcommand{\cTse}{\mathcal{T}^{\mbox{\tiny$\scriptstyle\searrow$}}}
\newcommand{\cPne}{\mathcal{P}^{\mbox{\tiny$\scriptstyle\nearrow$}}}
\newcommand{\cPnw}{\mathcal{P}^{\mbox{\tiny$\scriptstyle\nwarrow$}}}

\newcommand{\cPse}{\mathcal{P}^{\mbox{\tiny$\scriptstyle\searrow$}}}
\newcommand{\cQne}{\mathcal{Q}^{\mbox{\tiny$\scriptstyle\nearrow$}}}

\newcommand{\cQse}{\mathcal{Q}^{\mbox{\tiny$\scriptstyle\searrow$}}}

\newcommand{\cCnw}{\mathcal{C}^{\mbox{\tiny$\scriptstyle\nwarrow$}}}
\newcommand{\cCsw}{\mathcal{C}^{\mbox{\tiny$\scriptstyle\swarrow$}}}

\newcommand{\cCse}{\mathcal{C}^{\mbox{\tiny$\scriptstyle\searrow$}}}

\newcommand{\cSq}{\mathcal{S}q}
\newcommand{\Sq}{{Sq}}
\newcommand{\cCp}{\mathcal{C}p}
\newcommand{\Cp}{{Cp}}
\newcommand{\cA}{\mathcal{A}}
\newcommand{\cW}{\mathcal{W}}
\newcommand{\cM}{\mathcal{M}}
\newcommand{\cP}{\mathcal{P}}

\newcommand{\cT}{\mathcal{T}}

\newcommand{\cDCsw}{\mathcal{D}^{\mbox{\tiny$\scriptstyle\swarrow$}}}
\newcommand{\cDCnw}{\mathcal{D}^{\mbox{\tiny$\scriptstyle\nwarrow$}}}
\newcommand{\cDCse}{\mathcal{D}^{\mbox{\tiny$\scriptstyle\searrow$}}}
\newcommand{\cDCne}{\mathcal{D}^{\mbox{\tiny$\scriptstyle\nearrow$}}}
\newcommand{\cTSnw}{\mathcal{T}^{\mbox{\tiny$\scriptstyle\nwarrow$}}}
\newcommand{\cTSsw}{\mathcal{T}^{\mbox{\tiny$\scriptstyle\swarrow$}}}

\newcommand{\DCsw}{{D}^{\mbox{\tiny$\scriptstyle\swarrow$}}}
\newcommand{\tDCsw}{\widetilde{D}^{\mbox{\tiny$\scriptstyle\swarrow$}}}
\newcommand{\DCnw}{{D}^{\mbox{\tiny$\scriptstyle\nwarrow$}}}
\newcommand{\DCse}{{D}^{\mbox{\tiny$\scriptstyle\searrow$}}}
\newcommand{\DCne}{{D}^{\mbox{\tiny$\scriptstyle\nearrow$}}}
\newcommand{\TSnw}{{T}^{\mbox{\tiny$\scriptstyle\nwarrow$}}}
\newcommand{\TSsw}{{T}^{\mbox{\tiny$\scriptstyle\swarrow$}}}

\newcommand{\tTSsw}{\widetilde{T}^{\mbox{\tiny$\scriptstyle\swarrow$}}}

\newcommand{\Qse}{{Q}^{\mbox{\tiny$\scriptstyle\searrow$}}}
\newcommand{\Qne}{{Q}^{\mbox{\tiny$\scriptstyle\nearrow$}}}

\author{Enrica Duchi}

\title[Square permutations and convex permutominoes]{A code for square permutations \\and convex permutominoes}
\affiliation{
  IRIF, Université Paris Diderot, France
}

\received{2019-3-30}

\revised{2019-7-16}

\accepted{2019-9-10}
\begin{document}

\publicationdetails{21}{2019}{2}{2}{5354}

\maketitle

\begin{abstract}
  In this article we consider square permutations, a natural
  subclass of permutations that can be defined in terms of either geometric conditions or pattern avoidance, and convex permutominoes, a related subclass of
  polyominoes. While these two classes of objects arise independently
  in various contexts, they play a natural role in the description of
  certain random horizontal and vertical convex  grid
  configurations.

  We propose a common approach to the enumeration of these two classes
  of objets that allows us to explain the known common form of
  their generating functions, and to derive new refined formulas and
  linear time random generation algorithms for these objects and the
  associated grid configurations.
\end{abstract}
\section{Introduction}
Square permutations and convex permutominoes   are natural subclasses of
permutations and polyominoes that were introduced independently in the
last ten years and have been shown to enjoy remarkably simple and
similar enumerative formulas: their respective generating functions
$\Sq(t)$ and $\Cp(t)$ with respect to the size are
\begin{align}\label{eq:sq}
  \Sq(t)&=\frac{t^2}{1-4t}\left(2+\frac{2t}{1-4t}\right)-\frac{4t^3}{(1-4t)^{3/2}}\\\label{eq:cp}
  \Cp(t)&=\frac{t^2}{1-4t}\left(2+\frac{2t}{1-4t}\right)-\frac{t^2}{(1-4t)^{3/2}}
\end{align}
Equivalently, the number $\Sq_n$ of square permutations with $n$ points and the number $\Cp_n$ of
convex permutominoes with size $n$ are respectively
\begin{align*}
  (n+2)\,2^{2n-5}&\,-\,4(2n-5){2n-6\choose n-3}\\
  (n+2)\,2^{2n-5}&\,-\,(2n-3){2n-4\choose n-2}
\end{align*}

These results were first obtained by Mansour and Severini~\cite{SeTu}
for $\Sq(t)$, then recovered by Duchi and Poulalhon~\cite{DP} and by
Albert \emph{et al} \cite{ALRVW}, while for $\Cp(t)$ they were first
obtained by Boldi \emph{et al} \cite{milanesi} and, independently, by
Disanto \emph{et al} \cite{DiFrPiRi}. Known proofs of these formulas
rely on writing recursive decompositions resulting into linear
equations with one catalytic variable that can be easily solved via
the kernel method. An explicit connection between the two classes of
objects was obtained by Bernini \emph{et al} \cite{jis}, resulting in
a composition relation of the form $\Cp(t)=ISq(t,2)$ where $ISq(t,u)$
is the generating function of certain indecomposable square
permutations counted by their size and number of free fixed points
(see Section~\ref{sec:DefinitionsMotivations} for definitions and
precise statements).

While relatively simple, none of these known proofs explain, as far as
we know, the common shape of the formulas, nor its particular form as
a difference between an asymptotically dominant rational term and a
simple subdominant algebraic term.

By introducing an
injective encoding of square permutations and of convex permutominoes
into a same class of marked words $\cM$, we provide here an
explanation of the common form of the two formulas and generalize them to take into
account natural parameters extending the Narayana refinement for
Catalan numbers.

The rest of the paper is organized as follows. In
Section~\ref{sec:DefinitionsMotivations} we provide an original
motivation for the definitions of convex permutominoes and square
permutations. In Section~\ref{sec:EncodingMainResults} we first
present our encoding and introduce the announced class of marked words
$\cM$ containing the code words and interpreting the rational part of
the formula. We then state our main results and some consequences for
random generation.  The remaining sections are dedicated to the proofs
of the results: In Section~\ref{sec:NonCodingSquares} and
Section~\ref{sec:NonCodingIndec} and
Section~\ref{sec:NonCodingPermutominoes} we describe the subsets of
words of $\cM$ which do not encode square permutations, indecomposable
square permutations, and convex permutominoes respectively. This
allows to conclude the proof of the results stated in
Section~\ref{sec:EncodingMainResults}.

\section{Definitions and motivations}\label{sec:DefinitionsMotivations}
\subsection{Square permutations}
Let  $S$ be a set of points in the plane.  A point of $S$ is
a \emph{upper-right record} if there is no point in $S$ that is both
over it  and to its right. In other terms,
\begin{itemize}
\item $(x,y)\in S$ is a upper-right record of $S$ if for all
  $(x',y')\in S$, either $x'\leq x$ or $y'\leq y$ (or both).
\end{itemize}
One defines similarly \emph{upper-left}, \emph{bottom-left} and
\emph{bottom-right records}. Finally a point of $S$ that is a
record in one of these four directions is called a \emph{exterior
  point} of $S$, otherwise it is an \emph{interior point}.

In the rest of the paper we identify a permutation $\sigma$ of the set
$\{1,\ldots,n\}$ with its geometric representation, the set of points
$R(\sigma)=\{(i,\sigma(i))\mid i=1,\dots,n\}$. In particular for
simplicity, \emph{the point $\sigma(i)$ of the permutation $\sigma$} stands
for \emph{the point $(i, \sigma(i))$ of the geometric representation of
the permutation $\sigma$}.

Observe that the records
of $R(\sigma)$ correspond to the records of $\sigma$ in the usual
combinatorial sense: upper-right records of $R(\sigma)$ are right-left
maxima of $\sigma$, bottom-right records are right-left minima, and so
on for other types of records.

\begin{figure}
   \begin{center}
     \includegraphics[scale=.9]{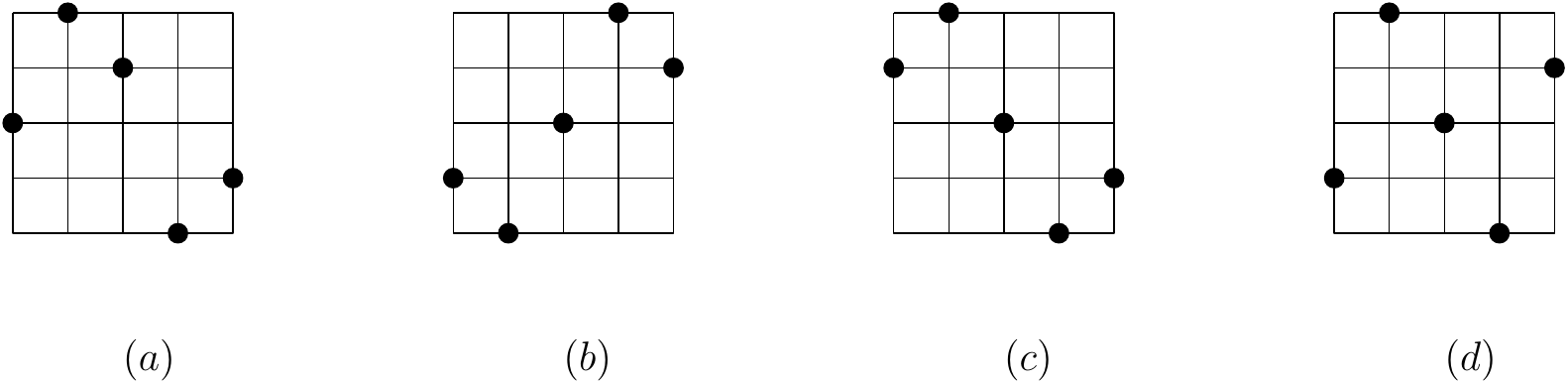}
   \end{center}
  \caption{$(a)$ The square permutation $3,5,4,1,2$. $(b)$ A square
    permutation with an upper-left record which is also a bottom-right
    record. $(c)$ A square permutation with an upper-right record which is
    also a bottom-left record. $(d)$ A permutation that is not
    square.}
  \label{Fig:Squares}
\end{figure}

Then we have the following definition  (see Figure~\ref{Fig:Squares}):
\begin{itemize}
\item A permutation is a \emph{square permutation} if its geometric
  representation has no interior point.
\end{itemize}
The terminology square permutation arises from the following path
representation: given a permutation, connect the successive upper-left
records from left to right to form the \emph{left-upper path} of $\sigma$, and define accordingly the right-upper path, right-lower path and left-lower path. Then an equivalent definition is the following:
\begin{itemize}
\item A \emph{square permutation} is a permutation such that each point belongs to the left-upper,  right-upper,   right-lower
 or  right-upper path.
\end{itemize}

The following two subclasses of square permutations play an important role in our approach: 
\begin{itemize} 
\item A {\em
  triangular permutation} is a square permutation such that each point
belongs either to the left-upper path, the right-upper path, or the right-lower
path (see Figure~\ref{Fig:gridperm}~(a)).

\item A {\em parallel permutation} is a square permutation such that all
points belong to the left-upper or the right-lower path (see Figure~\ref{Fig:gridperm}~(b)).
\end{itemize}

\begin{figure}
\centerline{
     \includegraphics[scale=1]{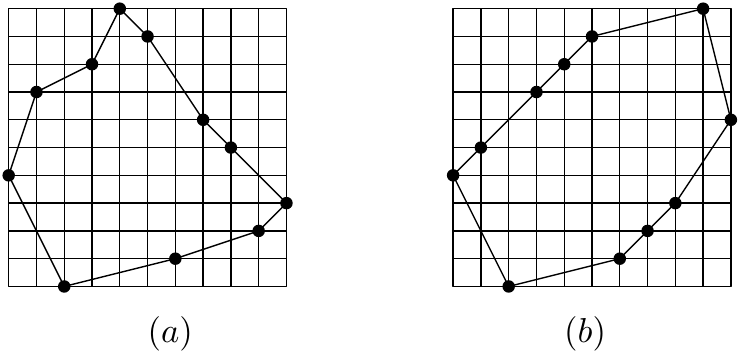}
}   \caption{(a) A triangular permutation. (b) A parallel permutation.}
   \label{Fig:gridperm}
\end{figure}

The {\emph size of a permutation} is its number of points. We let
$\mathcal{S}_q$ denote the class of square permutations and
${\mathcal{S}_q}_n$ denote the number of square permutations of size
$n$. We let 
$\cT$
denote the class of triangular permutations and
$\cT_n$
denote the number of triangular permutations of size $n$. We
let 
$\cP$
denote the class of parallel permutations and
$\cP_n$
denote the number of parallel permutations of size $n$.

\begin{remark}
  We would like to point out that:
\begin{itemize}
  \item Square permutations can be characterized in terms of pattern avoidance, in particular Waton proved~\cite{Waton} that they are all the permutations avoiding the following sixteen patterns of length five:
  $$
 \left\{ \begin{array}{llllllll}
    14325 & 14352 & 15324 & 15342 &24315 & 24351 & 25314 &25341 \\
    41325 & 41352 & 42315 & 42351 & 51324 & 51342 & 52314 & 52341 \\
  \end{array}\right\}
 $$
\item Triangular permutations are all the permutations avoiding the following four patterns of length four:
  $$
  \left\{
  \begin{array}{llll}
3214 & 3241 & 4213 & 4231
  \end{array}
  \right\}
  $$
\item Parallel permutations are permutations avoiding $321$, indeed they can be seen as two interleaved increasing sequences of points.
 \end{itemize} 
\end{remark}

\subsection{Convex permutominoes}
Like for square permutations, we use an approach in terms of points in the plane to  introduce convexity in the context of
self-avoiding polygons:
\begin{itemize}
\item A closed self-avoiding walk on the square lattice is a
  \emph{convex polygon} if, as set of points, it has no interior point. 
\end{itemize}
This notion of convexity for polygon is equivalent to the one used in
the context of polyomino enumeration (see \cite{BMBM}, and more
precisely \cite{BMG} for convex polygons).

The \emph{turnpoints} of a polygon on the square lattice are the
vertices of the lattice at which the underlying walk changes
direction. The \emph{sides} of a polygon are the sequences of steps
between turnpoints: any polygon on the square lattice has the same
number of sides and turnpoints and they are both even. The \emph{upper
  walk} of a convex polygon is the sub-walk of the polygon starting
from the highest of its leftmost points with a horizontal step and
ending at the highest of its rightmost points with a horizontal step.
An \emph{upper side} of a polygon is a horizontal side that belongs to
the upper walk of the polygon. One defines similarly the \emph{left
  walk} and the \emph{left sides} of a polygon.

We are interested in the subclass of polygons that are ``generic''
in the sense that they do not have two sides in the same column or
line, and ``reduced'' in the sense that each line or column that they
intersect contains a side:
\begin{itemize}
\item A polygon $P$ on the square lattice is a \emph{permutomino} of size
  $n$ if each line of abscissa or ordinate $i$ contains exactly one
  side of $P$ for $i=0,\ldots, n$.
\end{itemize}
Finally a \emph{convex permutomino} is a permutomino which is convex (see Figure~\ref{Fig:Permutomino}).
 Equivalently:
  \begin{itemize}
  \item A convex permutomino is a permutomino such that each point of the underlying walk is a record.
  \end{itemize}
  In fact it is clearly enough to check that each turnpoint is a record. 
\begin{figure}
 
   \begin{center}
     \includegraphics[scale=.8]{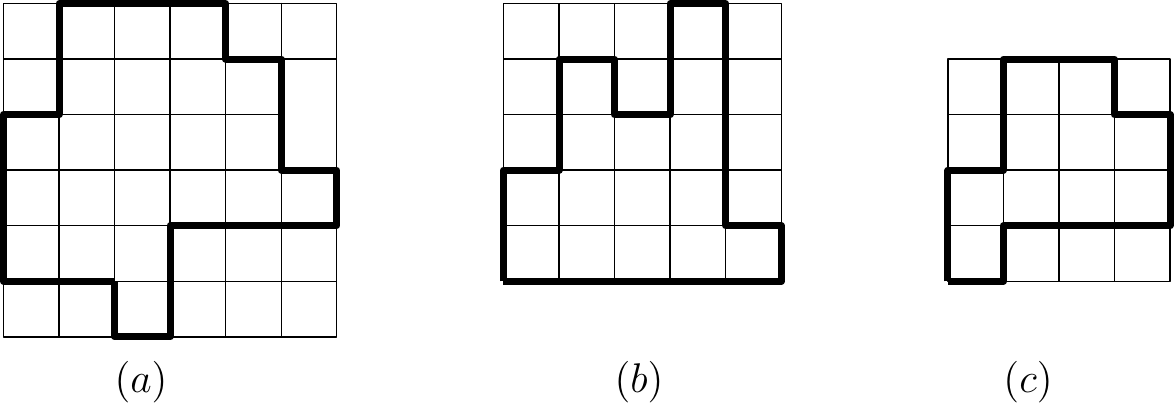}
   \end{center}
   \vspace{-1em}
  \caption{$(a)$ A convex permutomino of size $7$. $(b)$ A permutomino that is not convex. $(c)$ A convex polygon that is not a permutomino.}\vspace{-1em}
  \label{Fig:Permutomino}
\end{figure} 
As for square permutations, two subclasses of convex polygons will be of particular interest:
\begin{itemize} 
\item A {\em directed convex permutomino} is a permutomino such that each turnpoint is  either a left-upper, a right-upper, or a right-lower record (see Figure~\ref{Fig:permutominoesbis}~(a)).

\item A {\em parallelogram permutomino} is a permutomino such that each turnpoint  
is either a left-upper or a right-lower record (see Figure~\ref{Fig:permutominoesbis}~(b)).
\end{itemize}

\begin{figure}[tb]
\begin{center}
\includegraphics[width=.4\textwidth]{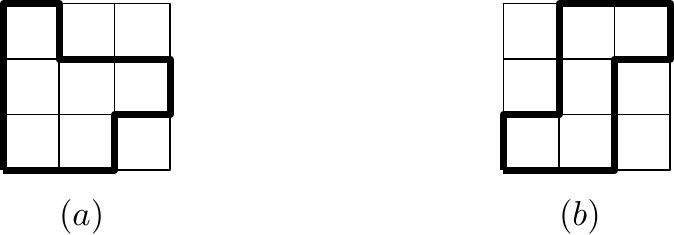}
\end{center}
\caption{$(a)$ a directed convex permutomino; $(b)$ a   parallelogram permutomino.} \label{Fig:permutominoesbis}
\end{figure}
\begin{remark}
  We would like to point out that permutominoes can be viewed as collections of unit squares lattice cells glued together such that the interior is connected, that is, they are a subclass of polyominoes. 
\end{remark}
We let $\mathcal{C}_p$ denote the class of convex permutominoes and  ${\mathcal{C}_p}_n$ the number of convex permutominoes of size $n$. We  let $\mathcal{D}_p$ denote the class of directed convex permutominoes and ${\mathcal{D}_p}_n$ denote the number of directed convex permutominoes of size $n$. We  let  $\mathcal{P}_p$ be the class of parallelogram permutominoes and ${\mathcal{P}_p}_n$ the number of parallelogram permutominoes of size $n$.
Finally, we let  $D_p(t)$ and $P_p(t)$ denote respectively  the generating functions of triangular and parallelogram permutominoes.
\paragraph{Square permutations versus convex permutominoes}
It is known that, besides the similarities between Formulas~\eqref{eq:sq} and~\eqref{eq:cp} for square permutations and convex permutominoes, the parallel between these two classes  extends also  to their subclasses. In particular, we have that:
$$ T(t)=B(t) \quad \textrm{ and } \quad D_p(t)=\frac{B(t)}{2}
\qquad \textrm{ where } \quad
  B(t)=\frac{1}{\sqrt{1-4t}}
$$
  is the generating function of central binomial coefficients, and
  \[
 P(t)=P_p(t)=C(t)
 \qquad\textrm{  where }\quad
    C(t)=\frac{1-2t-\sqrt{1-4t}}{2t}
    \]
is the generating function of Catalan numbers.

\subsection{A link with random grid configurations.}
The relevance of the number $Sq_n$ of square permutations and $Cp_n$ of
convex permutominoes is highlighted by the following propositions
(see Figure~\ref{Fig:PandTPGrille}):

\begin{proposition}
The number $E_{M,N}(n)$ of configurations of $n$ exterior points on
a $M\times N$ grid satisfies
\[
E_{M,N}(n)\mathop{\sim}_{M,N\to\infty}Sq_{n}\cdot {M\choose n}{N\choose n}.
\]
\end{proposition}
\begin{proof}
  Observe that the probability to have a non generic configuration, that is a configuration with two or more points belonging to the same column or row, goes to $0$. Indeed  $M$ and $N$ tends to infinity unlike the number of points $n$.
  Given a generic configuration of $n$ exterior points on a grid $M\times N$ it can be uniquely decomposed into a square permutation giving its reduced shape and a selection of $n$ of the $N$ columns and $n$ of the $M$ rows that are used by the points. 
  \end{proof}

A similar result holds for convex polygons on the square lattice with
a fixed number of turnpoints.
\begin{proposition}
The number $C_{M,N}(n)$ of convex polygons with $2n$ turnpoints on a
$M\times N$ grid satisfies
\[
C_{M,N}(n)\mathop{\sim}_{M,N\to\infty}Cp_{n}\cdot {M\choose n}{N\choose n}.
\]
\end{proposition}
\begin{proof}
  Observe that the probability to have a non generic polygon, with two or more turnpoints belonging to the same column or row, goes to $0$. Indeed  $M$ and $N$ tends to infinity unlike the number of turnpoints $n$.
  Given a generic convex polygon it can be uniquely decomposed into a convex permutomino giving its reduced shape and a selection of $n$ of the $N$ columns and $n$ of the $M$ rows that are used by the turnpoints. 
\end{proof}

\begin{figure}
   \begin{center}
     \includegraphics[scale=.7]{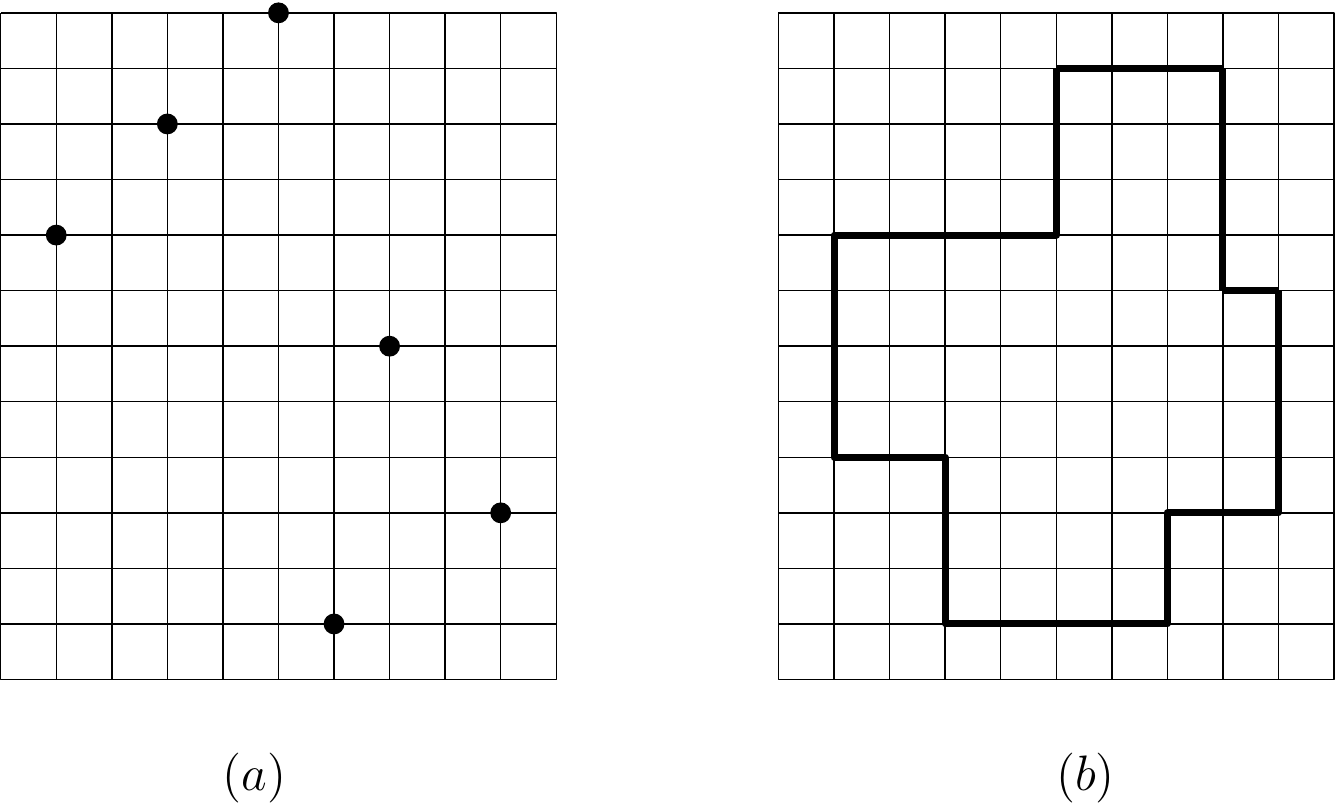}
   \end{center}
  \caption{$(a)$ A set of 6 exterior points on a grid of size $M\times N$. $(b)$ A convex polygon with 12 turnpoints on a grid of size $M\times N$.}
  \label{Fig:PandTPGrille}
\end{figure} 

\subsection{Convex permutominoes and colored square permutations}
A permutation $\sigma$ of $\{1,\ldots,n\}$ is \emph{decomposable} if
there exist two permutations $\pi$ of $\{1,\ldots,k\}$ and $\pi'$ of
$\{1,\ldots,\ell\}$ with $k+\ell=n$ such that $\sigma$ in vector
notation is the permutation $\pi_1,\pi_2,\ldots,\pi_k,
\pi'_1+k,\pi'_2+k,\ldots,\pi'_\ell+k$.  It is \emph{co-decomposable}
if there exist two permutations $\pi$ of $\{1,\ldots,k\}$ and $\pi'$
of $\{1,\ldots,\ell\}$ with $k+\ell=n$ such that $\sigma$ is the
permutation
$\pi_1+\ell,\pi_2+\ell,\ldots,\pi_k+\ell,\pi'_1,\pi'_2,\ldots,\pi'_\ell$. A
permutation is \emph{indecomposable} if it is not decomposable,
\emph{co-indecomposable} if it is not co-decomposable, and \emph{fully
  indecomposable} if it is both indecomposable and co-indecomposable.

A \emph{fixed point} of a permutation $\sigma$ is an entry $i$ such that
$\sigma(i)=i$. It is a \emph{free fixed point} if the corresponding point is
not a bottom-left or upper-right record. Let $f(\sigma)$ denote the number of free fixed points of $\sigma$.

A \emph{colored permutation} is a permutation with a (possibly empty) subset of its free fixed points that are colored.
In particular a colored square permutation is a colored permutation whose underlying permutation is square, and we
identify standard permutations with colored permutations without colored fixed point.
\begin{theorem}(Bernini \emph{et al} \cite{jis})\label{Berni}
  There is a bijection $\phi$ between
  \begin{itemize}
  \item convex permutominoes of size $n$, and
  \item colored square permutations of size $n$ that are co-indecomposable.
  \end{itemize}
\end{theorem}
To obtain the colored permutation associated to a convex permutomino
$P$, one should first color the bottom turnpoint on the leftmost edge
of $P$ in black, and then alternatively color all other turnpoints of
$P$ in white or black, following the boundary of $P$ in clockwise
order, so that adjacent turnpoints have different colors. Then
$\phi(P)$ is the permutation represented by black points, with fixed
points that belong to the upper walk of $P$ as colored fixed points
(resp. belonging to the lower boundary of $P$ as non-colored fixed
points). See an example in Figure~\ref{Fig:BijSquarePermu}.

\begin{figure}
   \begin{center}
     \includegraphics[scale=.9]{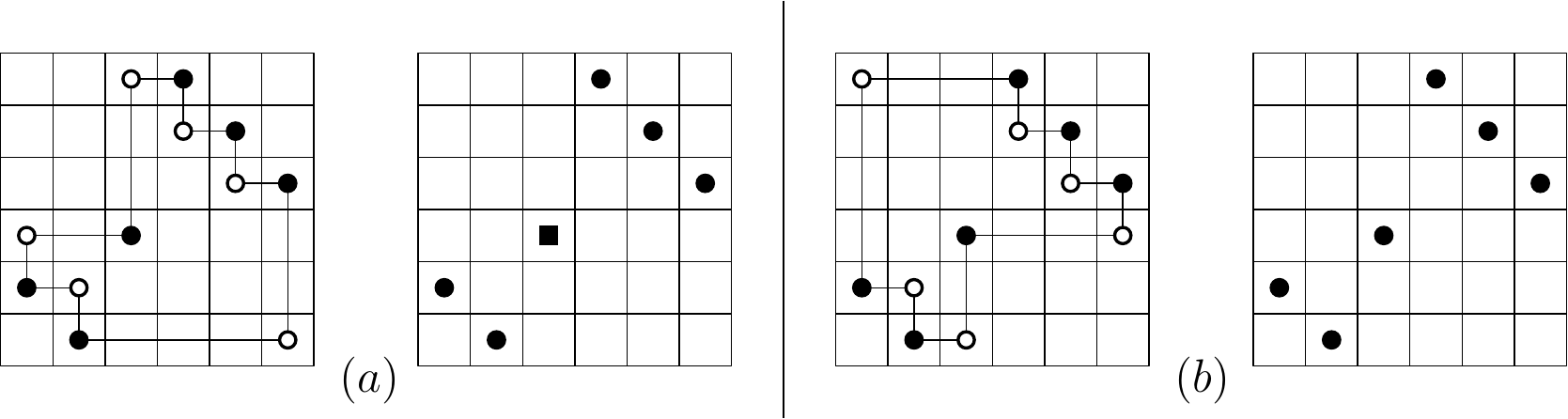}
   \end{center}
   \caption{$(a)$ A permutomino and its corresponding square permutation with a colored fixed point (represented with a little square). $(b)$ A permutomino and its corresponding  square permutation with a non-colored fixed point.}
  \label{Fig:BijSquarePermu}
\end{figure} 
As announced in the introduction, this theorem leads to a relation
between the generating function of convex permutominoes with respect to the size and
the generating function of co-indecomposable square permutations with respect to the
size and number of free fixed points. However, this relation does not
explain the similarity between Formulas~\eqref{eq:sq}
and~\eqref{eq:cp}.

\section{Main results}\label{sec:EncodingMainResults}

\subsection{The horizontal-vertical encoding}
We now define the horizontal-vertical encoding of a colored square 
permutation. We will later study this encoding for three subsets of
colored square permutations: square permutations (that is, colored square
permutations without colored points), indecomposable square
permutations, and colored co-indecomposable square permutations (which are in bijection with convex permutominoes as previously discussed).

Let us say that a point $(x,y)$ in a set $S$ is an \emph{upper point} if it is
a upper-left or a upper-right record of $S$, and it is a \emph{left point} if it is a
upper-left or bottom-left record of $S$.
Given a colored permutation $\sigma$, we denote by
$H(\sigma)=Xu_2 \ldots u_{n-1}X$ the \emph{horizontal profile} of $\sigma$,
where
$$
u_i= 
\begin{cases}U,& \text{if }  \sigma(i) \text{ is an upper point that is not colored;}\\
D,& \textrm{otherwise}\\
\end{cases}
$$ with $i \in \{2,\ldots, n-1\}$, and the letter $X$ codes for the
horizontal projection of extremal points of $\sigma$, and we denote by
$V(\sigma)= Y v_2 \ldots v_{n-1}Y$ the \emph{vertical profile} of $\sigma$,  where
$$
v_j=
\begin{cases}
L,& \text{if } (\sigma^{-1}(j),j) \text{ is a left point that is not colored;}\\ R,& \text{otherwise;}
\end{cases}
$$ with $j \in \{2,\ldots, n-1\}$, and the letter $Y$ codes for the
vertical projection of extremal points of $\sigma$.

\begin{figure}
   \begin{center}
     \includegraphics[scale=.8]{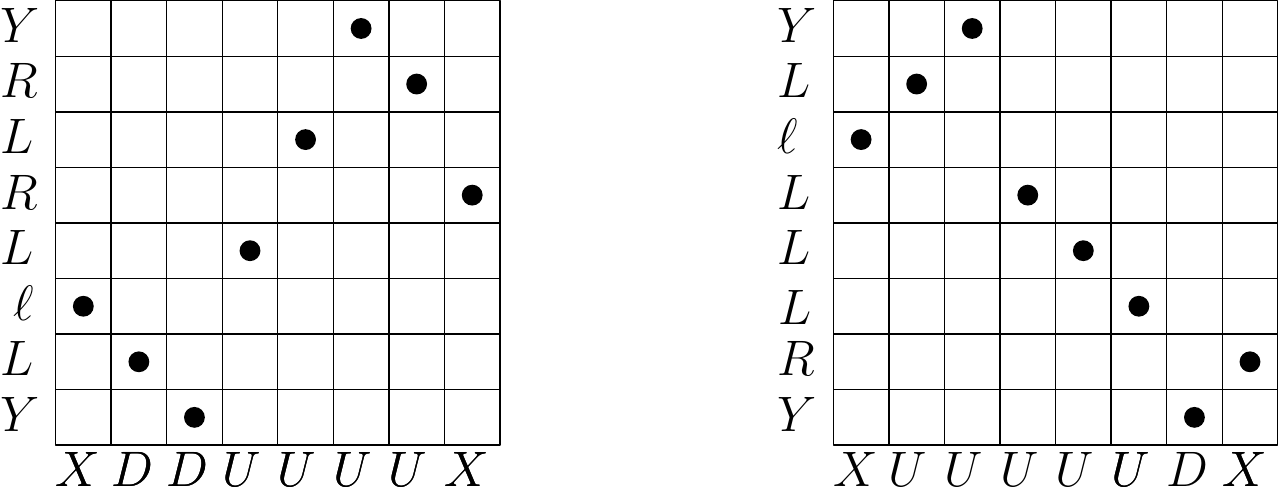}
   \end{center}
   \vspace{-1em}
   \caption{Two square permutations of size 8 whose marked biwords are respectively
     $(X,Y)(D,L)(D,L)(U,L)(U,R)(U,L)(U,R)(X,Y)$ with $m=3$
     and $(X,Y)(U,R)(U,L)(U,L)(U,L)(U,L)(D,L)(X,Y)$ with $m=6$.}
  \label{Fig:CodeSquare}
  \vspace{-1em}
\end{figure}

Finally let the {\em code} of a permutation $\sigma$ be the triple
$(H(\sigma),V(\sigma),\sigma(1))$: we will see later that the encoding
$\sigma \mapsto (H(\sigma),V(\sigma),\sigma(1))$ is indeed injective for the families of permutations we are interested in.

In order to state our results let us consider the following set
of words:
  Let $\cW=\cA^*$ denote the set of (bi)words on the alphabet
  $\cA=\{U,D\}\times\{L,R\}$, and let $\cM$ denote the set of marked words $(w,m)$ consisting of
  \begin{itemize}
  \item a word $w=(u_1,v_1)\cdots(u_n,v_n)\in\{(X,Y)\}\cdot\cW\cdot\{(X,Y)\}$
  \item and a mark $m$ with $1\leq m\leq n$ and $v_m\in\{L,Y\}$.
  \end{itemize}
  
Observe then that the triple $(H(\sigma),V(\sigma),\sigma(1))$ can
alternatively be viewed as a marked biword $(w,m)$ of $\cW$ by taking
$w=(X,Y)(u_2,v_2)\ldots(u_{n-1},v_{n-1})(X,Y)$ and $m=\sigma(1)$
where, by construction, the mark $m$ is on a letter $L$ or $X$ of the
vertical profile (see Figure~\ref{Fig:CodeSquare} for an example,
where the marked letter $L$ is represented by $\ell$).

The main contribution of this paper is  to prove that the above  horizontal/vertical encoding of square permutations and  of convex permutominoes  by words of $\cM$, is  injective and to describe the complementary sets of \emph{non-coding words}, that is, the elements of $\cM$ that are not codewords in both cases:
\begin{theorem}\label{thm:main}
  The horizontal/vertical encoding defines bijections 
  \begin{align*}
    \cSq&\equiv\cM\;
    -\;\cTSsw\cdot\{(D,L),(D,R)\}\cdot\cW\cdot\{(X,Y)\}\;
    -\;\cTSnw\cdot\{(U,R),(D,L)\}\cdot\cW\cdot\{(X,Y)\}\\ \\
    \cCp&\equiv\cM\;
    -\;\cDCsw\cdot\{(D,L)\}\cdot\cW\cdot\{(X,Y)\}\;
    -\;\cDCnw{}^+\cdot\cW\cdot\{(X,Y)\}
  \end{align*}
  where $\cTSsw$ and $\cTSnw$ are 
  languages encoding variants of
  \emph{triangular permutations} (namely permutations obtained from triangular permutations by a $180^o$ or $90^o$ rotation respectively), while $\cDCsw$ and $\cDCnw{}^+$ are 
  languages encoding
  variants of \emph{directed convex permutominoes} (namely permutations obtained respectively from colored triangular permutations by a $180^o$ rotation or from indecomposable triangular permutations by a $90^o$ rotation respectively).
\end{theorem}
\subsection{Enumerative consequences and random sampling}\label{MainResults}

  Let now $M(t;x,y)$ and $W(t;x,y)$ be respectively the generating functions of 
$\cM$ and $\cW$ with respect to the length (var. $t$), number of $U$ and $X$
    (var. $x$) and number of $L$ and $Y$ (var. $y$). Then

    \begin{align*}
      W(t;x,y)=\frac{1}{(1-(1+x)(1+y)t)}
    \end{align*}
    and, upon dealing separately
    with the case $m\in\{1,n\}$ where the mark is on a letter $Y$ from the
    case $1<m<n$ where the mark is on a letter $L$, we have that
    \begin{align*}
      M(t;x,y)=2\cdot txy\cdot W(t;x,y)\cdot txy
      +(txy)\cdot W(t;x,y)\cdot t(1+x)y\cdot W(t;x,y)\cdot txy.
  \end{align*}
    In particular $M(t,1,1)$ gives an \emph{a priori unrelated}
    combinatorial interpretation of the dominant rational term in
    Formula~\eqref{eq:sq} and~\eqref{eq:cp} since
    \begin{align*}
      W(t;1,1)=\frac{1}{1-4t}\quad\textrm{ and }\quad M(t;1,1)=\frac{t^2}{1-4t}\left(2+\frac{2t}{1-4t}\right).
    \end{align*}

From Theorem~\ref{thm:main} it already appears that the two sets of words that do not encode objects belonging to $\cSq$ and to  $\cCp$ according to their \emph{horizontal/vertical encoding}, that we  call \emph{non-coding words} for  $\cSq$ and for  $\cCp$, have a similar
structure. In fact the analogy between the two results goes further
since the languages  $\cTSsw$ and $\cTSnw$, $\cDCsw$, $\cDCnw{}^+$ have
essentially the same univariate generating functions:
\begin{align}\label{eq:allalgebraic}
\TSsw(t)=\TSnw(t)=(2\DCsw(t)+1)t=2\DCnw{}^+(t)-t=\frac{t}{\sqrt{1-4t}}
\end{align}
and these results together with Theorem~\ref{thm:main} immediately
imply Formulas~\eqref{eq:sq} and~\eqref{eq:cp}.

The evaluations~\eqref{eq:allalgebraic} can be obtained from algebraic
decompositions of the respective classes of objects, but we choose
also to obtain them by re-using the same horizontal/vertical
encoding to relate all these four generating functions to the generating series
$N(t;x,y)$ of Narayana numbers, defined by the equation
  \begin{align*}
    N(t;x,y)=t\cdot(1+xN(t;x,y))(1+yN(t;x,y)).
  \end{align*}
  From this analysis we obtain the refinements of
  Formulas~\eqref{eq:sq} and~\eqref{eq:cp}:
\begin{corollary}
  Let $Sq(t,x,y)$ be the generating function of square
  permutations with respect to the size ($t$), the number of upper
  points ($x$) and the number of left points ($y$), and let $Cp(t;x,y)$ be the generating function of convex permutominoes with respect to
  the size ($t$), the number of upper sides ($x$) and the number of
  left sides ($y$).  Then
  \begin{align*}
    \Sq(t;x,y)=M(t;x,y)
    &-\;\tTSsw(t,x,y)\cdot t\cdot (y+1)\cdot W(t,x,y)\cdot txy\;\\
    &-\;\TSnw(t,x,y)\cdot t\cdot (x+y)\cdot W(t,x,y)\cdot txy\\
    \Cp(t;x,y)=M(t;x,y)
    &-\;\DCsw(t,x,y)\cdot t\cdot y\cdot W(t,x,y)\cdot txy\;\\
    &-\;\DCnw{}^+(t,x,y)\cdot W(t,x,y)\cdot txy
  \end{align*}
  where the auxiliary series are explicit rational functions of $x$, $y$ and the Narayana function evaluations
  $N(t;x,y)$ and $N(t;xy,1)$. In particular
  \begin{align*}
  \Tnw(t;x,y)&=\frac{xyN(t;x,y)}{(1-xyN(t;x,y))(1+(x+y-xy)N(t;x,y))}\\
  \tTSsw(t;x,y)&=\frac{xyN(t;xy,1)}{(1-yN(t;xy,1))(1+N(t;xy,1))}.
  \end{align*}
\end{corollary}
\begin{remark}
  Observe that, when specializing $x=y=1$, we obtain  $\tTSsw(t;1,1)=\Tsw(t;1,1)$.
  \end{remark}
In the proof of Theorem~\ref{thm:main} we provide a linear time decoding procedure that, given a word in
$\mathcal{M}$,  outputs the corresponding square permutation or convex
permutomino. Consequently, since Formulas~(\ref{eq:sq})--(\ref{eq:cp}) immediately imply that
$\Sq_n$ and $\Cp_n$ behave asymptotically as the 
number of words of size $n$ in $\cM$, we also have the following corollary:
\begin{corollary}
  There is a random sampling algorithm to generate uniform random
  square permutations with $n$ points or uniform random convex
  permutominoes of size $n$ in expected time linear in $n$.
\end{corollary}
This result implies in turn  that there is an algorithm to
produce in linear time generic random horizontally and vertically
convex configurations of points and self-avoiding polygons on a large
grid.

We would like to point out that the properties of large random square permutations are studied by Borga and Slivken in~\cite{JE}, where they proved that their shape is typically rectangular.

\section{A decoding algorithm for square permutations}\label{sec:NonCodingSquares}

We first give a detailed description of the decoding algorithm for square permutations (without colors), in the other cases the decoding algorithm will be similar.

We already pointed out that the set of triples coding square
permutations of size $n$ forms a subset of $\cM$. However, this subset
is not easy to describe directly. We provide an algorithmic
description via the following decoding algorithm.

Let us take an element of $\cM$ of length $n$, viewed as a triple
$(u,v,\ell)$. We are going to construct a partial permutation
$\sigma=\sigma(1)\ldots\sigma(k)$, with $k\le n$. The idea is to insert
each point on the grid from left to right as long as a compatible
decoding is possible.  The rules to construct $\sigma$ depend on the
letters of $u$ and $v$: the letter $L$ (resp.\ $R$, $U$, $D$) means
that the point of the permutation we are constructing belongs to a
left path (resp.\ right, lower, upper path) of the corresponding
permutation, while the letters $X$ or $Y$ means that this
point is extremal. Then (see Figure~\ref{ParRec})\,:
\begin{itemize}
\item Write $u$ (resp.\ $v$) along the upper edge (resp.\ left edge) of the  grid $n \times n$ starting from the left (resp.\ starting from the bottom).
\item Start reading the word $u$: $u_1$ is necessarily an $X$ and it
  should correspond to the leftmost point of the permutation. Accordingly
  insert a point $\sigma(1)$ in the leftmost column and in row $\ell$
  (the row of the marked letter of $v$).
\item Continue reading each letter $u_i$ of  $u$, for $i \ge 2$:
  \end{itemize}
\begin{enumerate}
\item If $i=n$ then insert the point $\sigma(n)$ in the unique remaining empty row  and end the algorithm.
\item If $u_i= U$ and we  have not yet inserted the highest point of the permutation, then take the first row $j$  above $\sigma(i-1)$   that is labeled with  $L$ or  $Y$ and insert the point  $\sigma(i)=j$.    
\item 
  If $u_i= U$ and we have already inserted the highest point of the
  permutation, there are two cases:
  \begin{itemize}
  \item if $\sigma(1)\ldots\sigma(i-1)$ is contained in the grid $(i-1)\times (i-1)$ attached to the upper-left
    corner of the grid $n\times n$ then the insertion is valid only if the row $n-i$ is labeled  with $L$ or $Y$, and then $\sigma(i)=n-i$;
    otherwise the insertion is  invalid and the algorithm stops (see the first case in
    Figure~\ref{ParRec}, where the invalid insertion is represented with
    a white bullet);
  \item otherwise  take the first row $j$ under $\sigma(i-1)$ that is labeled with $R$ or $Y$ and insert the point $\sigma(i)=j$.
  \end{itemize}
\item If $u_i= D$ and we have not yet inserted the lowest point of the permutation, then take the first row $j$ that is labeled with  $L$ or  $Y$ under $\sigma(i-1)$ and insert the  point $\sigma(i)=j$:
 if $\sigma(1)\ldots\sigma(i-1)$ is contained in the grid $(i-1)\times (i-1)$ attached to the upper-left corner of the grid $n\times n$ and $j=i$ then the insertion is invalid and the algorithm stops (see the second case in  Figure~\ref{ParRec}).
\item If $u_i= D$ and we have already inserted the lowest point of
  the permutation, then take the first row $j$ that is labeled with
  $R$ or $Y$ above $\sigma(i-1)$ and insert the point
  $\sigma(i)=j$: 
 if $\sigma(1)\ldots\sigma(i-1)$ is contained in
  the grid $(i-1)\times (i-1)$ attached to the bottom-left corner of
  the grid $n\times n$ then the insertion is invalid and the algorithm
  stops (see the third case in Figure~\ref{ParRec}).
\end{enumerate}

\begin{figure}
   \begin{center}
     \includegraphics[scale=.9]{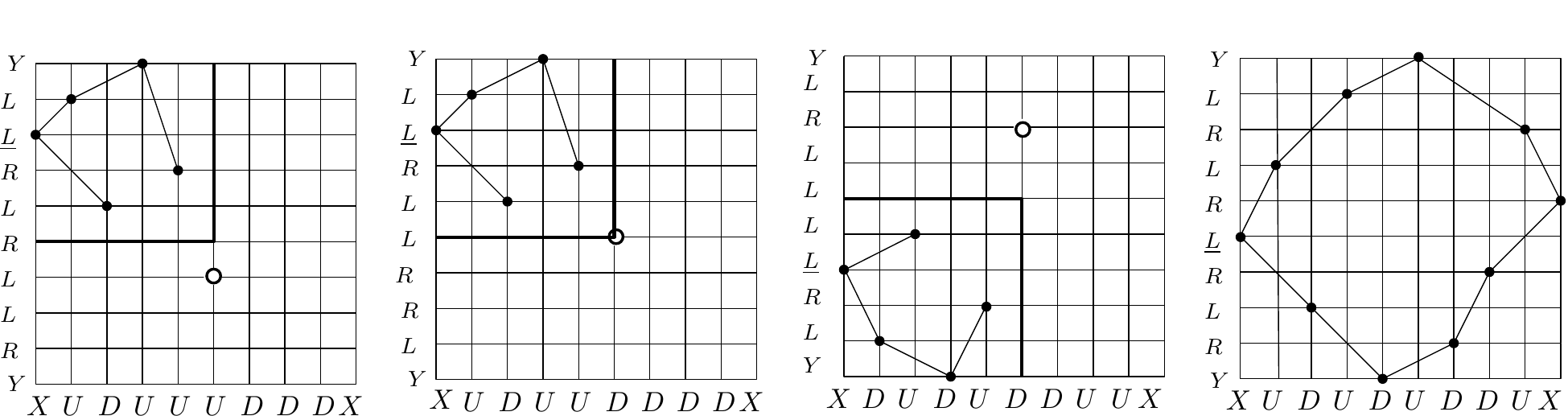}
   \end{center}
   \vspace{-1em}
  \caption{Four examples of triples $(u,v,\ell)$, only the fourth represents an encoding of a square permutation. }
  \label{ParRec}
  \vspace{-1em}
\end{figure} 
Let us suppose that the algorithm ends after reading $i$ letters of
$u$ and $v$, $i \in \{1, \ldots, n\}$.  
Then we can state the
following proposition\,:
\begin{proposition}\label{words}
 Given an element $(w,\ell)$ of $\cM$ of size $n$ viewed as triple
 $(u, v,\ell)$ the algorithm stops and yields:
\begin{itemize}
\item (in case of success) a permutation $\sigma(1) \ldots
\sigma(n) \in\mathcal{S}q_n$ whose code is $(w,\ell)$, or
\item (in Cases 3 and 4 above) a triple $(\rho,(u_i,v_{n-i}),(u'',v''))$,
  where $\rho=\mathrm{Standard}(\sigma(1) \ldots \sigma(i-1))\in
  \Tnw_{i-1}$, with $2\leq i\leq n$, is the standardization of
  $\sigma$ on $\{1,\ldots,i-1\}$, i.e., $\rho_l=\sigma(l)-n+i-1$ with
  $l\in \{1,\ldots,i-1\}$, $(u_i,v_{n-i})\in\{(U,R),(D,L)\}$, and
  $(u'',v'')$ is a pair of words of length $n-i-1$ where
  $u''=u_{i+1}\ldots u_{n-1}X$ and $v''=Yv_{2} \ldots v_{n-i-1}$, or
\item (in Case 5 above) a triple $(\rho,(u_i,v_i),(u'',v''))$, where
  $\rho=\sigma(1) \ldots \sigma(i-1) \in \Tsw_{i-1} $ with $2\leq
  i\leq n$, $(u_i,v_i)\in \{(D,L),(D,R)\}$, and $(u'',v'')$ is a pair
  of words of length $n-i-1$, where $u''=u_{i+1}\ldots u_{n-1}X$ and
  $v''=v_{i+1} \ldots v_{n-1}Y$.
\end{itemize}
\end{proposition}

\begin{proof}[Proof of Proposition~\ref{words}]
By construction the permutation obtained by applying the algorithm is
uniquely determined and it has no interior points. Indeed, the first
point $\sigma_1$ is determined by the marked letter and Case 2
(resp.\ 3,4,5) of the previous algorithm uniquely describes the
insertion of a point in the left-upper path (resp.\ right-upper path,
left-lower path, right-lower path) of $\sigma$. We just need to prove
that row $j$ in Case 2 (resp.\ 3, 4, 5) does exist. Suppose
that it does not, then for Case 2 (resp.\ 4) it means that the
point at the row labeled by $Y$, i.e., the higher (resp.\ lower) point
of the permutation it has already been inserted, thus contradicting
the hypothesis that it has not. For Case 3 (resp.\ 5) we have two
possibilities\,:
\begin{itemize}
 \item the lower (resp.\ higher ) point of the permutation has already
   been inserted, therefore the leftmost left-upper and left-lower
   paths have been constructed, i.e., there is no $L$ row free, thus
   contradicting the hypothesis that there is no $R$ or $Y$ row free;
\item the lower (resp.\ higher ) point of the permutation has not been
  inserted, therefore the lower (resp.\ higher) row, labeled by $Y$,
  is still free, thus contradicting the hypothesis.
\end{itemize} 
Now we have to distinguish the two ways in which the algorithm can stop\,:
\begin{enumerate}
\item \label{one}$i=n$.  In this case we obtain by construction a
  permutation in $\mathcal{S}q_n$ and $(u,v,\ell)$ is the code of
  $\sigma$. 
\item $i<n$. In this case we obtain from the previous algorithm a
  partial permutation $\sigma$ of size $i-1$ and a pair of words
  $(u'',v'')$ of length $n-i-2$. If the algorithm stops in Case~3 or~4
  then the partial permutation $\sigma$ of the grid $n\times n$ is
  contained in the grid $(i-1) \times (i-1)$ on the upper-left corner.
  Since $\sigma$ is contained in the $i-1$ upper row, the lowest row has
  not been touched and the construction of a right-lower path has not
  started: $\rho=\mathrm{Standard}(\sigma)$ is a rotated
  triangular permutation $\in \Tnw_{i-1}$. If the algorithm
  stops in Case~5, $\rho=\sigma_1 \ldots \sigma_{i-1}$ is
  directly a permutation of $\{1,\ldots,i-1\}$ and similarly as above
  $\rho \in \Tse_{i-1}$ since the construction of a right-upper path
  has not started.
\end{enumerate}
\end{proof}
This yields a bijection corresponding to the first statement in Theorem~\ref{thm:main}:
  \begin{align*}
    \cM
    &\equiv 
    \cSq
    \;
    +\;\cTSsw\cdot\{(D,L),(D,R)\}\cdot\cW\cdot\{(X,Y)\}\;
    +\;\cTSnw\cdot\{(U,R),(D,L)\}\cdot\cW\cdot\{(X,Y)\}
  \end{align*}
  In turn this bijection implies an equation for refined generating series:
  \begin{align*}
    M(t;x,y)
    =
    \Sq(t;x,y)
    \;&
    +\;\tTSsw(t;x,y)\cdot t\cdot (1+y)\cdot W(t;x,y)\cdot txy\;\\
    &+\;\TSnw(t;x,y)\cdot t\cdot (x+y)\cdot W(t;x,y)\cdot txy
  \end{align*}
  where $\tTSsw(t;x,y)=(txy)^{-1}\TSsw_1(t;x,y)-1$ is a modified
  version of the bivariate generating series of south west triangular
  permutations in which the rightmost point contributes to the
  parameter $x$ only if it is maximal, and $\TSsw_1(t;x,y)$ denotes
  the generating series of south west triangular permutations ending
  with a maximal point. More generally we use the subscript 1 to
  indicate that a subset of triangular or parallel permutations having
  a unique (extremal) point on the trivial face.
\begin{corollary}
\begin{equation}\label{triangWordsbis}
\{(w,n)\in \cM\}\;\equiv\; (\bullet\cTse) + (\bullet\cPse) \times \{(U,R),(D,L)\}\times \mathcal{W}\times\{(X,Y)\},
\end{equation}
\end{corollary}
\begin{proof}
  Indeed, observe that the permutations of $\Tse$ of size $n-1$ are exactly
the permutations of $\Sq$ that can be obtained from permutations of
$\cSq$ of size $n$ that start with the maximal value $n$. Moreover, the permutations of $\cPse$ of size $k-1$ are exactly
the permutations  that can be obtained from permutations of
$\Tne$ of size $k$ that start with the maximal value. Then by applying  Theorem~\ref{thm:main} we obtain the result. 
\end{proof}
Consequently,
\begin{align}
txy\cdot W(t;x,y)\cdot txy=(\Tse_1(t;x,y)-txy)+\Pse_1(t;x,y)\cdot t(x+y)\cdot W(t;x,y)\cdot txy.
\end{align}
These relations are enough to compute the univariate generating
functions (with $x=y=1$): indeed by symmetry of the various classes of
triangular permutations and by knowing that $\cPse$ are
counted by Catalan numbers we directly obtain a proof of 
Formula~\ref{eq:sq}.

In order to obtain the refined generating series a bit more
information is necessary because the four symmetry classes of
triangular permutations have different refined generating series: the
symmetry along the co-diagonal exchanges the parameter counted by $x$
and $y$, so that $\Tse_1(t;x,y)=\Tse_1(t;y,x)$ and
$\Tnw_1(t;x,y)=\Tnw_1(t;y,x)$ while $\Tne_1(t;x,y)=\Tsw_1(t;y,x)$, but
symmetries do not imply such direct relations between these three
classes of series.

Forcing the initial point to be in row one yields a second equation that
allows to compute the refined generating series $\Tne_1(t;x,y)$ (and $\Tsw_1(t;x,y)$):
\begin{corollary}
\begin{equation}\label{triangWordster}
\{(w,1)\in \cM\}\;\equiv\; (\bullet\cTne) + (\bullet {\cPne}) \times \{(D,L),(D,R)\}\times \mathcal{W}\times\{(X,Y)\},
\end{equation}
\end{corollary}
Consequently,
\begin{align}
txy\cdot W(t;x,y)\cdot txy=(\Tne_1(t;x,y)-txy)+\tPne_1(t;x,y)\cdot t(1+y)\cdot W(t;x,y)\cdot txy,
\end{align}
where $\tPne_1(t;x,y)=\tPse_1(t;xy,1)$.

However, the approach has to be modified in order to deal with the last
series $\Tnw(x,y)$: by symmetry with respect to the main diagonal this
amounts to counting elements of $\cTse$ with respect to the number of
bottom points and right points; however note that it is not possible
to simply change the encoding into $D$ for lower points and $R$ for
right points while keeping the left to right decoding procedure
because the later becomes ambiguous.  Instead one can rely on the
fact that $\cTnw$ can be described in terms of $\cPnw$ and the set
$\cCnw$ of co-indecomposable elements of $\cTnw$, and similarly for
$\cTse_1$ and $\cCse$: more precisely we have
\begin{equation}
  \cTnw\equiv \cCnw \times (1+\cPnw), \qquad \cTse_1\equiv\bullet+\cPse_1\times\cCse.
\end{equation}
These decompositions yield the refined identities:
\begin{align}\label{eq:indec}
  \Tnw(t;x,y)&=\Cnw(t;x,y)(1+\tPnw(t;x,y))=\Cnw(t;x,y)\Pse_1(t;x,y)\\
  \Tse_1(t;x,y)&=txy+\Pse_1(t;x,y)\tCse(t;x,y)
\end{align}

Then due to the fact that co-indecomposable permutations have no
points that are simultaneously upper and left points, apart from their
uppermost and leftmost points (which are distinct), we have
$\Cnw(t;x,y)=\tCse(txy;\frac1y,\frac1x)xy$.

As a conclusion,
\[
\Tnw(t;x,y)=\Pse_1(t;x,y)\frac{\Tse_1(txy,\frac1y,\frac1x)xy-txy}{\Pse_1(txy,\frac1y,\frac1x)}
\]

In order to finish the computation we need to identify the various refinements of the Catalan generating series. Let $N(t;x,y)$ be the Narayana generating series, the 
unique formal power series solution of the equation
\[
N(t;x,y)=t(1+xN(t;x,y))(1+yN(t;x,y)).
\]
Then
\[
\Pnw_1(t;x,y)=\frac{xyN(t;x,y)}{1+(x+y-xy)N(t;x,y)}.
\]
Observe also that $N(txy,\frac1y,\frac1x)=xyN(t;x,y)$ so that
$\tPne_1(t;x,y)=\Pnw_1(t;xy,1)=txy(1+xyN(t;xy,1))$, so that
all expressions are rational in terms of $x$, $y$, $N(t;x,y)$ and
$N(t;xy,1)$.

In particular we obtain that
\begin{align}\label{for:Tnwxy}
  \Tnw(t;x,y)&=\frac{xyN(t;x,y)}{(1-xyN(t;x,y))(1+(x+y+xy)N(t;x,y))}\\\label{for:Tswxy}
  \tTSsw(t;x,y)&=\frac{xyN(t;xy,1)}{(1-yN(t;xy,1))(1+N(t;xy,1))}
\end{align}
which together with 
\begin{align*}
    \Sq(t;x,y)
    =    M(t;x,y)
    \;&-\;\tTSsw(t;x,y)\cdot t\cdot (1+y)\cdot W(t;x,y)\cdot txy\;\\
      &-\; \TSnw(t;x,y)\cdot t\cdot (x+y)\cdot W(t;x,y)\cdot txy
  \end{align*}
gives the result.

\bigskip
\begin{corollary}
Uniform random square permutations can be generated in
expected linear time.
\end{corollary}
\begin{proof} To generate a permutation of size $n$,
one generates triples $(u, v,\ell)$ from $\cM_n$ and applies the
decoding algorithm until a square permutation of size $n$ is
obtained. Each generation and decoding step takes linear time in
$n$. The probability of rejection goes to zero as $n$ goes to infinity
since it corresponds to partially decoded permutations that are
asymptotically negligible with respect to the total number of square
permutations.
\end{proof}

\section{A decoding algorithm for fully indecomposable square permutations}\label{sec:NonCodingIndec}
By using the previous encoding for square permutations we can
specialize the bijection of Theorem~\ref{thm:main} to the class of
indecomposable square permutations, co-indecomposable square
permutations or fully indecomposable square permutations (\emph{i.e.}
square permutations that are both indecomposable and
co-indecomposable). The idea is to adapt the previous decoding
algorithm in order to remove triples that do not correspond to objects
of these class: each resulting algorithm is similar to the decoding
algorithm for square permutations, with slightly different stopping
conditions. We present here the case of fully indecomposable square
permutations (see Figure~\ref{Fig:RecPartInCoIn} for an example):
\begin{itemize}
  \item in Case 2 above  one should stop as soon as
  $\sigma_1\ldots\sigma_{i-1}$ is contained in the grid
  $(i-1)\times(i-1)$ attached to the lower-left corner, regardless of
  the other conditions;
 
\item in Cases 3 and 4 above, one should stop as soon as
  $\sigma_1\ldots\sigma_{i-1}$ is contained in the grid
  $(i-1)\times(i-1)$ attached to the upper-left corner, regardless of
  the other conditions.
\end{itemize}
\begin{figure}
   \begin{center}
     \includegraphics[scale=.7]{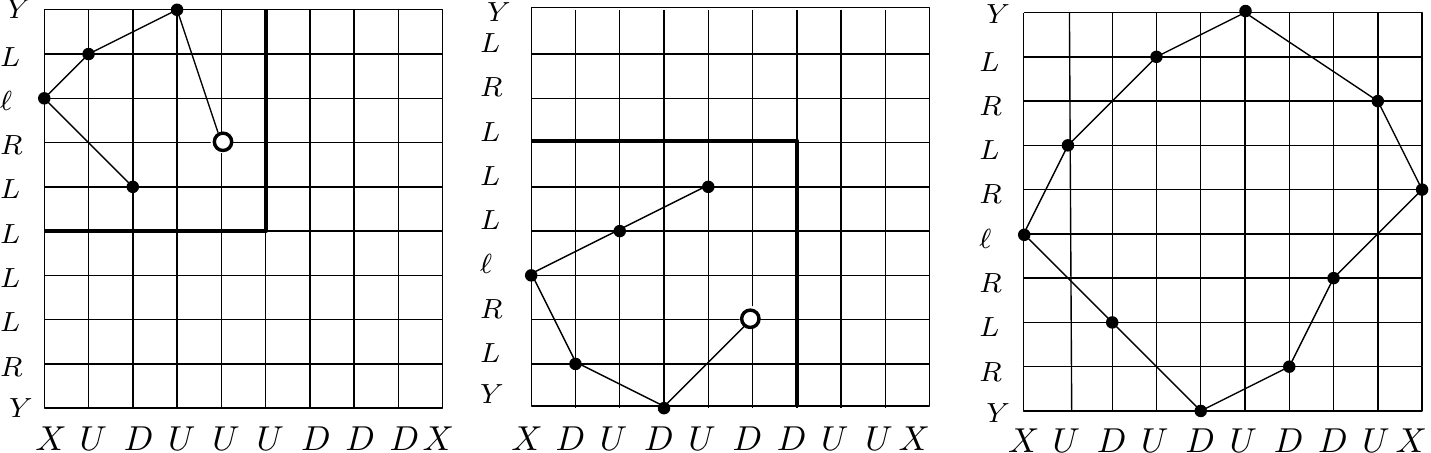}
   \end{center}
   \vspace{-1em}
  \caption{Three examples of triples $(u,v,\ell)$. Only the  third one  gives a square permutation that is indecomposable and co-indecomposable.}
  \label{Fig:RecPartInCoIn}
  \vspace{-1em}
\end{figure} 
Let $\mathcal{F}$ be the class of fully indecomposable square permutations, then from this decoding algorithm we obtain the following bijection:
  \begin{align*}
    \cM
    &\equiv 
    \mathcal{F}
    \;
    +\;\cCsw\cdot\cW\cdot\{(X,Y)\}\;
    +\;\cCnw\cdot\cW\cdot\{(X,Y)\}
  \end{align*}
  and consequently the following equation for its refined generating function:
  \begin{align*}
    M(t;x,y)
    =
    {F}(t;x,y)
    \;&
    +\;\tCsw{}(t;x,y)\cdot W(t;x,y)\cdot txy\;\\
    &+\;\Cnw{}(t;x,y)\cdot W(t;x,y)\cdot txy
  \end{align*}
  where $\tCsw(t;x,y)$ and $\Cnw(t;x,y)$ can obtained by Formula~\eqref{for:Tswxy}--\eqref{for:Tnwxy} via Equation~\eqref{eq:indec}. In particular for $x=y=1$ we see that the (known) can be written in a similar form as $Sq(t)$ and $Cp(t)$:
    \[
    F(t;1,1)=\frac{t^2}{1-4t}\left(1+\frac{2t}{1-4t}\right)-\frac{t^2}{(1-4t)^{3/2}}\\
    \]

\section{The encoding for convex permutominoes and a decoding algorithm}\label{sec:NonCodingPermutominoes}
\begin{figure}
   \begin{center}
     \includegraphics[scale=.7]{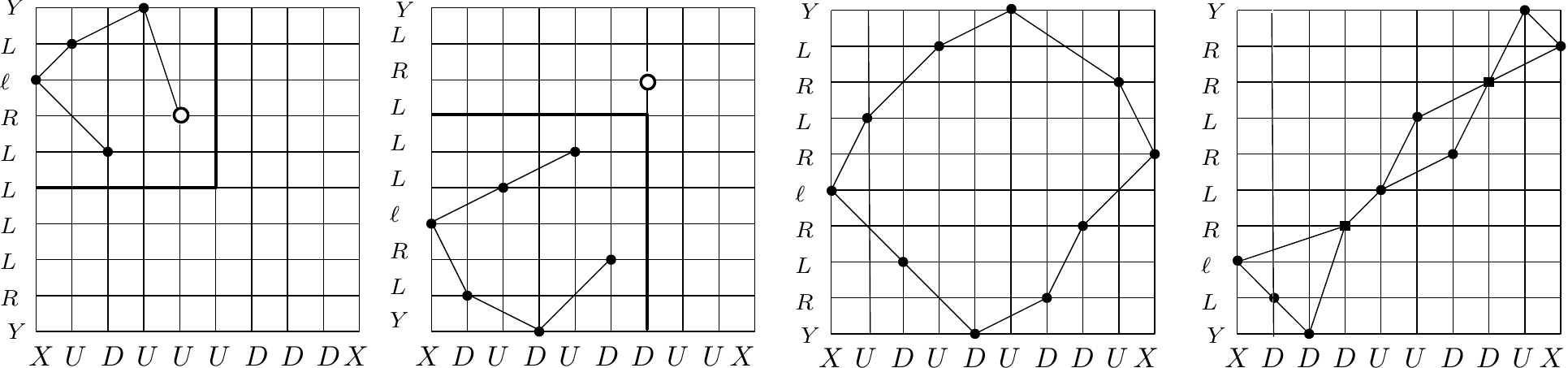}
   \end{center}
   \vspace{-1em}
  \caption{Four examples of triples $(u,v,\ell)$. The third one and the fourth one belong to the class of  colored co-indecomposable square permutations.}
  \label{Fig:RecPart2}
  \vspace{-1em}
\end{figure}

The decoding algorithm for convex permutominoes, viewed as colored
co-indecomposable square permutations, is similar to the decoding
algorithm for square permutations, with slightly different stopping conditions (see Figure~\ref{Fig:RecPart2} for an example):
\begin{itemize}
\item in Cases 3 and 4 above, one should stop as soon as
  $\sigma_1\ldots\sigma_{i-1}$ is contained in the grid
  $(i-1)\times(i-1)$ attached to the upper-left corner, regardless of
  the other conditions, to eliminate co-decomposable permutations;
\item in Case 5 above, one should declare the insertion invalid only if
  $j\neq i$, to allow for colored fixed points.
\end{itemize}

As a result, we obtain the bijection
\begin{align*}
  \cM \;&\equiv\;\cCp
  +\cDCnw{}^+\cdot W\cdot\{(X,Y)\}  +\cDCsw\cdot\{(D,L)\}\cdot W\cdot\{(X,Y)\}
\end{align*}
where $\cDCsw$ denotes the subclass of colored permutations that can
be obtained from colored square permutations ending with a maximal
value by removing this maximal value, and $\cDCnw{}^+$ denotes the
subclass of co-indecomposable
colored permutations that can be obtained from  colored co-indecomposable 
square permutations ending with a minimal value by removing this last value
and standardizing the permutation.

In order to deal with $\cDCnw{}^+$ and $\cDCsw$   we use again that the
symmetric permutation classes $\cDCse{}^+$ and $\cDCne$ can be obtained as
adapted special cases of the above construction of $\cCp$:
\begin{itemize}
\item $\cDCse{}^+$, co-indecomposable colored square permutations starting with the maximal value ($\sigma_1=n$ for a permutation of size $n$):
  \begin{align*}
    \{(w,n)\in\cM\} &\equiv\;
    (\bullet\cDCse{}^+) 
    +(\bullet\cQse{}^+)\cdot W\cdot\{(X,Y)\} +\bullet\cdot
    \{(D,L),(D,R)\}\cdot W\cdot\{(X,Y)\}
  \end{align*}
where $\cQse{}^+$ denotes the class of permutations obtained from co-indecomposable permutations of $\cCp$ that start with the maximal value and end with the minimal value by removing these extremal values.

\item $\cDCne$,  colored square permutations starting with the minimal value ($\sigma\in\cSq$ s.t. $\sigma_1=1$), are obtained when the decoding is applied to
  pairs $(w,1)$:
\begin{align*}
  \{(w,1)\in\cM\}&\equiv\;(\bullet\cDCne)+(\bullet(1+\cQne))\cdot\{(D,L)\}\cdot W\cdot\{(X,Y)\}
\end{align*}
where $\cQne$ denotes the class of permutations that are obtained from
colored permutations of $\cCp$ that start with the minimal value and
end with the maximal value by removing these extremal values.
\end{itemize}

Finally the classes $\cQse$ and $\cQne$ are again simple variants of the
class of permutations avoiding the pattern 123, that for which various
algebraic decompositions or bijective constructions are available.
\paragraph{Generating functions}

Taking into account the $U$, $X$, $L$ and $Y$s, the above bijection translate into generating function relations. For convex permutominoes:
 \begin{align*}
    M(t;x,y)
    =
    \Cp(t;x,y)
    \;&
    +\;\tDCsw(t;x,y)\cdot ty\cdot W(t;x,y)\cdot txy\;\\
    &+\;\DCnw{}^+(t;x,y)\cdot W(t;x,y)\cdot txy.
  \end{align*}
For the set $\DCse$, and  $\cDCne{}^+$,
the bijection yields
\begin{align*}
  txy\cdot W(t;x,y)\cdot txy&=\DCne_1(t;x,y)+\Qne(t;x,y)\cdot ty\cdot W(t;x,y)\cdot txy\\
    txy\cdot W(t;x,y)\cdot txy
    &=\DCse_1{}^+(t;x,y)+txy\cdot t(1+y)\cdot W(t;x,y)\cdot txy
        +\Qse{}^+(t;x,y)\cdot W(t;x,y)\cdot txy
\end{align*}

The rest of the computation is analogous to the case of square
permutations using symmetries, and leads again to a parametrization in
terms of the same refined Narayana series $N(t;x,y)$ and $N(t;xy,1)$.

\acknowledgements
\label{sec:ack}
We warmly thank the referees for their constructive suggestions as
well as Gilles Schaeffer for his help during the preparation of this
paper.

\label{sec:biblio}

\end{document}